\documentclass[12pt]{article}
\usepackage{diagbox}
\usepackage{mathtools}
\usepackage{bbm}
\usepackage{latexsym}
\usepackage{epsfig}
\usepackage{amsmath,amsthm,amssymb,enumerate}

\usepackage[a-1b]{pdfx}
\usepackage{hyperref}

\usepackage{color}

\def\nn{\nonumber}
\def\a{\alpha}  \def\d{\delta} 
\def\e{\varepsilon} \def\f{\phi} \def\F{{\Phi}}  \def\g{\gamma}
\def\G{\Gamma}  \def\k{\kappa}
\def\z{\zeta} \def\th{\theta}    \def\l{\lambda}
 \def\m{\mu}  \def\p{\pi}
\def\r{\rho}  \def\s{\sigma} 
\def\t{\tau}

\def\bx{{\bf x}}

\def\bal{\boldsymbol \alpha}
\newtheorem{theorem}{Theorem}
\newtheorem{lemma}[theorem]{Lemma}

\newtheorem{claim}{Claim}

\newcommand{\wh}[1]{\widehat{#1}}

\newcommand{\brac}[1]{\left(#1\right)}

\newcommand{\bfrac}[2]{\left(\frac{#1}{#2}\right)}

\def\cE{{\cal E}}

\newcommand{\set}[1]{\left\{#1\right\}}
\def\sm{\setminus}

\def\dd{\text{d}}
\def\E{\mathbb{E}}

\def\Pr{\mathbb{P}}

\def\cF{{\cal F}}
\newcommand{\ignore}[1]{}

\def\cE{{\mathcal E}}
\def\cF{{\mathcal F}}

\def\cM{{\mathcal M}}

\def\cY{{\mathcal Y}}

\newcommand{\beq}[2]{\begin{equation}\label{#1}#2\end{equation}}
\newcommand{\mults}[1]{\begin{multline*}#1\end{multline*}}

\def\nn{\nonumber}

\usepackage{tikz}
\usetikzlibrary{decorations.pathmorphing}
\usetikzlibrary{positioning}
\usetikzlibrary{arrows,automata}
\usetikzlibrary{shapes.misc}
\usetikzlibrary{backgrounds}
\usetikzlibrary{arrows,shapes}

\begin{document}
\author{Patrick Bennett\thanks{Department of Mathematics, Western Michigan University, Kalamazoo MI 49008-5248, Research supported in part by Simons Foundation Grant \#848648. Email: patrick.bennett@wmich.edu.}\and Alan Frieze\thanks{Department of Mathematical Sciences, Carnegie Mellon University, Pittsburgh PA15213, Research supported in part by NSF grant DMS2341774. Email: frieze@cmu.edu.}\and Wesley Pegden\thanks{Research supported in part by NSF grant DMS-1363136. Email: wes@math.cmu.edu.}}
\title{Expected cost in Combinatorial Optimization under color constraints}
\maketitle
\begin{abstract}
We present an average case model of classical problems in combinatorial optimization where there are color constraints. In all cases we seek some (spanning) sub-structure of a complete graph of minimum cost. The edges are randomly colored either red or blue. We bias against the red edges by placing a bound on the number of them that are allowed in our structure. This bound will be lower w.h.p. than what would occur without discrimination. We examine the effect of this bias on the minimum cost of a desired structure. We consider minimum cost spanning trees, shortest paths, minimum cost perfect matchings and the asymmetric traveling salesperson problem.
\end{abstract}
\section{Introduction}
We present an average case model of classical problems in combinatorial optimization where there are color constraints. In all cases we seek some (spanning) sub-structure of a complete graph of minimum cost. The edges are colored either red or blue. We bias against the red edges by placing a bound on the number of them that are allowed in our structure. This bound will be lower than what would occur with high probability without discrimination. We examine the effect of this bias on the minimum cost of a desired structure.

The analysis of randomly colored random graphs is a thriving area, see Chapter 15 of \cite{FK} for an overview of some of the best studied questions. In this paper we deal with random colorings within the context of optimization.

In this paper we use standard asymptotic notation, and asymptotics are as $n \rightarrow \infty$. We write $a_n \sim b_n$ if $a_n = (1+o(1))b_n$. We say a sequence of events $\cE_n$ occurs {\em with high probability (w.h.p.)} if $\Pr[\cE_n] \rightarrow 1$. 
\subsection{Spanning Trees}
First consider the complete graph $K_n$ where the edges are colored  either red or blue independently with probability 1/2. Each edge $e\in E(K_n)$ is given an independent uniform $[0,1]$ cost $X_e$. We let $E_R$ denote the set of red edges. We seek a minimum cost spanning tree that contains at most $\a n$ red edges, where $\a<1/2$. (We will assume that $\a n$ is an integer.) The cost $C(T)$ of a tree $T$ is given by $\sum_{e\in T}X_e$. Note that when $\a>1/2$ this restriction will be redundant w.h.p. Indeed, if we ignore the color constraint and just find the minimum cost spanning tree, w.h.p. it has at most $\a n$ red edges. Frieze \cite{Frieze85} proved that with no color constraint, the minimum cost spanning tree has cost $\z(3) = \sum_{k=1}^\infty k^{-3}$ w.h.p. In this context we prove 
\begin{theorem}\label{th1}
  Let $Z_{\a}=Z_\a(MST)$ denote the minimum cost of a spanning tree $T$ with $|E_R\cap T|\leq \a n$. If $\frac 1 4<\alpha< \frac 1 2$, we have
\[
\E Z_{\a}\sim \z(3)+(1-2\a)^2.
\]

More generally, let
\beq{eqn:fdef}
{
f(x)=\frac{1}{x}\sum_{k=1}^\infty\frac{k^{k-2}}{k!}(xe^{-x})^k=\begin{cases}1-\frac{x}{2}&x\le 1.\\ \frac{y}{x}\brac{1-\frac{y}{2}}&x>1.\end{cases}
}
Here, $y<1$ satisfies $ye^{-y}=xe^{-x}$.

If $\psi$ is the solution to 
\[
f(\psi)=2\a. 
\]
then for any $0<\alpha<\frac 1 2$, we have
\beq{parisi}{
\E Z_{\a}\sim \z(3)+\int_{x=0}^\psi f(x)\dd x-2\a\psi.
}
Furthermore, $Z_\a\sim\E Z_\a$ w.h.p.
\end{theorem}
The function $f(x)$ is asymptotically equal to $n^{-1}$ times the expected number of components in the random graph $G=G_{n,x/n}$. If $x<1$ then w.h.p. $G$ is a forest plus a few unicyclic components and so there are close to $n-|E(G)|$ components, where $|E(G)|\sim nx/2$ w.h.p. We do not prove what happens when $\a=1/2$, but letting $\a \nearrow 1/2$ into the above formula gives $\psi \rightarrow 0$ and $Z_\a\sim \z(3)$, which corresponds to the result of \cite{Frieze85}.

\subsection{Paths}
For this problem each edge $e$ of the complete graph (or digraph) is given a cost $X_e$ independently distributed as exponential mean 1, $EXP(1)$, i.e. $\Pr(X_e\geq x)=e^{-x}$ for $x\geq 0$. Edges are then independently colored red or blue with probability 1/2. Now for $0<\a < 1/2$ we compute minimum cost paths under the restriction that 
\beq{pathcolors}{
\text{no path can contain more than $L=\a\log n$ red edges.}
}
 (We assume that $\a\log n$ is an integer.) So now let $Z_\a=Z_{\a}(SP)$ denote the cost of the minimum cost path from 1 to 2, i.e.~two arbitrary fixed vertices. Then  
\begin{theorem}\label{th2}
\beq{RHS}{
Z_\a\sim\frac{\l\log n}{n} w.h.p.
}
where $0<\a<1/2$ and $\l=\l(\a)$ is the unique solution to
\[
 \frac{\lambda}{2}+\a\log\!\left(\frac{\lambda}{2\a}\right)+\a-1=0.
\]
\end{theorem}
Note that if we let $\a \nearrow 1/2$ then $\l \rightarrow 1$. This makes sense because the color constraint becomes essentially inactive, and Janson \cite{J2} proved that without the color constraint the minimum cost of a path from 1 to 2 is w.h.p. $\log n /n$. Similarly, if  $\a\searrow 0$ then $\l\to2$, reflecting the fact that we are only able to use blue edges. 

At the present moment we do not have a proof that $\E Z_{\a}\sim  \frac{\l\log n}{n}$, although we conjecture this to be true.
\subsection{Assignments}
For this problem each edge $e$ of the complete bipartite graph is given a cost $X_e$ independently distributed as exponential mean 1. Edges are then independently colored red or blue with probability 1/2. We seek a minimum cost perfect matching that contains at most $\a n$ red edges, where $0<\a<1/2$. (We assume that $\a n$ is an integer.) The cost $C(M)$ of a perfect  matching $M$ is given by $\sum_{e\in M}X_e$. In this context we prove 
\begin{theorem}\label{th3}
Let $Z_\a=Z_{\a}(PM)$ denote the minimum cost of a perfect matching $M$ with $|E_R\cap M|\leq \a n$, where $0<\a<1/2$. Let $\psi$ be the solution to
\[
\psi^{-1}e^{-\psi/2}\int_{x=0}^{\psi/2}\frac{e^x-1}{x}\dd x=\a.
\]
Then
\beq{PMcost}{
\E Z_{\a}\sim\frac{\p^2}{3}+\int_{x=0}^1\frac{e^{-\psi x/2}\log(1-x)}{x}\dd x-\a\psi.
}
Furthermore, $Z_\a\sim\E Z_\a$ w.h.p.
\end{theorem}
We note that as $\a\nearrow 1/2$, we have $\psi\rightarrow 0$. Noting that $\int_{x=0}^1\frac{\log(1-x)}{x}\dd x=-\p^2/6$, we get $\E Z_\a \rightarrow \p^2/6$. This is to be expected, since without the color constraint Aldous \cite{Aldous} proved that the expected minimum cost is $\p^2/6$.

\subsection{Asymmetric Traveling Salesperson Problem}
For this problem each edge $e$ of the complete digraph is given an independent exponential mean 1 cost $X_e$. Edges are then independently colored red or blue with probability 1/2.  We seek a minimum cost tour $T$ that contains at most $\a n$ red edges, where $\a<1/2$. The cost of a tour $T$ is given by $C(T)=\sum_{e\in T}X_e$. Karp \cite{Karp} described a {\em patching algorithm} that w.h.p. finds a tour of a cost $1+o(1)$ of the minimum. We use Karp's argument \cite{Karp} and prove
\begin{theorem}\label{th4}
Let $Z_{\a}(ATSP)$ denote the minimum cost of a Traveling Salesperson Tour $T$ with $|E_R\cap T|\leq \a n$. Then w.h.p.
\[
Z_{\a}(ATSP)\sim Z_{\a}(PM).
\]
 \end{theorem}
The above theorems are stated for specific distributions. It is easy to adapt them to more general distributions and this is done in Section \ref{distributions}.
\section{Proof of Theorem \ref{th1}}

First we verify the value of $f$ by evaluating the series in \eqref{eqn:fdef}. Let $W_0$ be the principal branch of the Lambert $W$ function, i.e. $W_0(z)$ is defined for $z \ge -1/e$, and returns the unique value $w \ge -1$ such that $we^w=z$. Thus, 
\begin{equation}\label{eqn:W}
W_0(-xe^{-x})= \begin{cases}
    -x, & x \le 1\\
    -y, & x > 1
    \end{cases}
\end{equation}
where $y<1$ satisfies $ye^{-y}=xe^{-x}$. The Taylor series for $W_0$ is well known:
\[
W_0(z) =  \sum_{k=1}^\infty\frac{(-1)^{k-1} k^{k-1}}{k!}z^k.
\]
Let
\[
S(z):=  \sum_{k=1}^\infty\frac{ k^{k-2}}{k!}z^k \qquad \text{ so } \qquad zS'(z) =  \sum_{k=1}^\infty\frac{ k^{k-1}}{k!}z^k = -W_0(-z).
\]
Now since $S(0)=0$ we have
\[
    S(z) = \int_0^z \frac{-W_0(-t)}{t} \dd t = \int_0^{-W_0(-z)} \frac{w}{we^{-w}}(1-w)e^{-w} \dd w = \int_0^{-W_0(-z)} (1-w) \;\dd w = -W_0(-z) - \frac12 W_0(-z)^2 
\]
where we made the substitution $w = -W_0(-t)$, or $t=we^{-w}$. Now by the series definition of $f(x)$ in \eqref{eqn:fdef} and using \eqref{eqn:W}, we get
\[
f(x) = \frac{1}{x}\sum_{k=1}^\infty\frac{k^{k-2}}{k!}(xe^{-x})^k = \frac{1}{x} S(xe^{-x}) = \frac 1x \left[-W_0(-xe^{-x}) - \frac12 W_0(-xe^{-x})^2\right] = \frac{1}{x} \cdot \begin{cases}
    x- \frac12 x^2, & x \le 1\\
    y-\frac 12 y^2, & x > 1
    \end{cases}
\]
which matches the final expression in \eqref{eqn:fdef}.
\subsection{Dual problem}
Let $X$ denote the set of spanning trees of $K_n$. Our optimization problem is
\begin{align}
\text{Minimise:}&\hspace{0.25in} C(T)\label{Trees}\\
\text{Subject to:}&\hspace{0.25in} T\in X\nn\\
&\hspace{0.25in}|E_R\cap T|\leq {\a n}.\label{T1}
\end{align}
Here we have a classical problem in combinatorial optimization, viz.~find the minimum cost set that is a basis of one matroid (the cycle matroid of $K_n$) and independent in a second matroid (the partition matroid defined by \eqref{T1}).

We analyse this through Lagrangian relaxation. Define the dual function $\f(\th)$ for $\th\geq 0$ by
\beq{Lagrange}{
\f(\th)=-\th\a n+L_n(\th).
}
where $L_n(\th)=\min_{T\in X}\set{C_\th(T)}$ where $C_\th(T)=\sum_{e\in  {T}}C_\th(e)$ and $C_{\th}(e)=X_e+\th1_{e\in E_R}$,

Let $\th^*$ be the value of $\th$ maximizing $\f(\th)$ over all $\th \ge 0$. It then follows from Edmonds' theorem \cite{E1} on the intersection of matroid polyhedra and Geoffrion's theorem on Lagrangian relaxation \cite{Ge} that
\[
\f(\th^*)=\max_{\th\geq 0}\f(\th)=\min_{T\in X}\set{ {C(T):\eqref{T1}\text{ holds}}},
\]
To see this, observe that problem \eqref{Trees} is actually a linear program, since as observed in \cite{Ge}, $\f$ finds an $\bx$ that solves the linear program where $T\in X$ is replaced by $\bx\in conv(X)$. And then \cite{E1} shows that in the case of matroid intersection adding \eqref{T1} does not create new fractional vertices i.e. the incidence vectors of trees with few enough red edges define the set of extreme points

We also know that w.h.p. the minimum cost of a spanning tree that only uses blue edges is asymptotic to $2\z(3)= 2.40411380632...$, see Beveridge, Frieze and McDiarmid \cite{BFM}. So, 
\beq{bounds}{
\z(3)\lesssim \f(\th^*)\lesssim 2\z(3).
}
\subsection{Analysis of $\f$}
Let the $U_e,e\in E(K_n)$ be a collection of independent and identically distributed (i.i.d.) uniform random variables on $[0,1]$ and let $B_e,e\in K_n$ be a collection of i.i.d. Bernoulli random variables with success probability 1/2. For each $e\in E(K_{n})$ set 
\[
W_{\th,e} = U_e + \theta B_e\text{ and }L_n(\th) = \min_{T \in X} \sum_{e \in T} W_{\th,e}.
\]
Then
\[
\E\f(\th)=-\th\a n+\E L_n(\th).
\]
Following Janson \cite{J1} we write
\begin{align*}
\E L_n(\th)& = \E\brac{ \sum_{e \in T_0} W_{\th,e}} = \E\brac{ \sum_{e \in T_0} \int_{x=0}^\infty1_{ { x < W_{\th,e}}}dx}=\int_{x=0}^\infty \E(\kappa(G_{n,\hat p(x)})-1) \dd x,
\end{align*}
where $\k(G)$ denotes the number of connected components of the graph $G=G_{n,\hat p(x)}$, and 
\[
\hat p(x) = \Pr(W_{\th,e} \leq x) =\frac12\Pr(U_e \leq x) + \frac12\Pr(U_e \leq x - \th) = \frac12\min\{x,1\} +\frac12\min\{(x-\theta)^+,1\}.
\]

\bigskip\noindent
 \textbf{Case 1.1.} $\theta\ge K/n$ for some large constant $K$. 
The dual function $\f(\th)=L_n(\th)-\th\a n$ and w.h.p. $L_n(\th)=O(1)$, independent of $\th$, given that we can just use blue edges for a tree. We have to maximise this over $\th\geq 0$. This rules out $\th\geq K/n$ for some constant $K>0$, as we know that $\max\f(\th)\geq \z(3)$.

\bigskip\noindent
\textbf{Case 1.2.} $0 < \theta < K/n$. Then, using $\th < 1$ we have
\[
\hat p(x) = \begin{cases}
x/2, & 0 < x < \theta, \\
x - \theta/2, & \theta < x < 1, \\
(1+x-\theta)/2, & 1 < x < 1 + \theta, \\
1, & x > 1+\theta.
\end{cases}
\]
Thus, since $\hat p(x)=1$ implies that $\k(G_{n,\hat p(x)})=1$,
\begin{align*}
\E L_n(\th) &= \int_{x=0}^{1+\th} \E(\kappa(G_{n,\hat p(x)})-1) \dd x\\
&=\int_0^\theta \E (\kappa(G_{n, x/2})) \dd x + \int_\theta^1 \E(\kappa(G_{n, 
x -\th/2})) \dd x + \int_1^{1+\theta} \E(\kappa(G_{n, (1+x-\th)/2})) \dd x-(1+\th)\\
            &= 2\int_0^{\theta/2} \E (\kappa(G_{n, x}))\dd x + \int_{\theta/2}^{1-\theta/2} \E (\kappa(G_{n, x}))\dd x + 2\int_{1-\theta/2}^{1} \E (\kappa(G_{n, x}))\dd x-1+O(1/n),\\
  &=  \int_0^{\theta/2} \E (\kappa(G_{n, x}))\dd x + \int_{0}^{1} \E (\kappa(G_{n, x}))\dd x -1+O(1/n)\\
&= \int_0^{\theta/2} \E (\kappa(G_{n, x}))\dd x + \z(3)+O(1/n),\qquad\text{uniformly for all }\th<K/n. 
\end{align*}
{\bf Explanation:} on the third line each of the three integrals is the same as one of the integrals in the previous line after making an appropriate substitution. On the fourth line we have used the fact that for $x \ge 1-O(1/n)$ we have $\E(\k(G_{n, x}))=1+o(1)$. On the final line we used the fact that $\int_{0}^{1} \E (\kappa(G_{n, x}))\dd p -1=\z(3)+O(1/n)$, due to Janson \cite{J1}.

\begin{claim} $\E (\kappa(G_{n, p})) = n f(np)+ o(n)$ uniformly for $p < K/n$.
\end{claim} 
\begin{proof}[Proof of claim] 
Now $\E (\kappa(G_{n, p})) = \E(A) + \E(B) + \E(C)$ where $A$ is the number of tree components on at most $\log^2 n$ vertices, and $B$ is the nontrees on at most $\log^2 n$ vertices, and $C$ is all components on more than $\log^2 n$ vertices. $B$ and $C$ will be small. Indeed, deterministically we have $C \le n/\log^2 n$. For $B$ we have
\begin{equation*}
    \E (B) \le  \sum_{k=1}^{\log^2 n} \binom nk k^{k-2}\binom k2  p^k (1-p)^{  k(n-k)} \le \sum_{k=1}^{\log^2 n} \frac{\brac{knpe^{-np}}^k}{k!}  (1-p)^{-\log^4 n} = O(\log^2 n),
\end{equation*}
where the final bound is because $ze^{-z} \le 1/e$ for all $z > 0$, $k! > (k/e)^k$, and $(1-p)^{-\log^4 n}=1+o(1)$.
Now for $A$ we have
\begin{align*}
    \E(A) &= \sum_{k=1}^{\log^2 n} \binom nk k^{k-2} p^{k-1} (1-p)^{k(n-k)+\binom k2 - k + 1}\\
    & = \brac{1 + O\brac{\frac{\log^4 n}{n}}} \frac{1}{p}\sum_{k=1}^{\log^2 n} \frac{k^{k-2} }{k!} (np)^{k} e^{-knp} = nf(np) + o(n). 
\end{align*}
\end{proof}
 In the next subsection we will show that $\E L_n(\th)$ is highly concentrated. Thus we have, on making the substitution $u=nx$ in the integral, 
\[
\f(\th) = - \th \a n+\int_0^{n\theta/2} f(u) \dd u +\z(3)  +o(1)
\]
w.h.p. 
Now differentiating the above w.r.t. $\th$ and setting it equal to 0, we see that $\th^*=2\psi/n$ where
$f(\psi) \sim 2\a.$ Thus w.h.p.
\[
\f(\th^*)=-2\a\psi+\int_0^{\psi} f(u) \dd u +\z(3)  +o(1),
\]
and Theorem \ref{th1} follows.
\subsection{Concentration }\label{concsec}
The goal of this section is to prove the following lemma.
\begin{lemma}\label{lm:conc}
For a fixed $0\leq \th\leq K/n$,
\beq{concphi}{
\Pr(|L_n(\th)- \E(L_n(\th))|\geq n^{-1/5}) =o(n^{-100}).
}
\end{lemma}
\begin{proof}
Recall that $\f(\th)=L_n(\th) - \th \a n$ (as defined in \eqref{Lagrange}). We will use Theorem 8.1.1 of Talagrand \cite{Tal}. Let $0<X_i\leq \xi,i=1,2,\ldots,N$ be a sequence of independent positive random variables and let $\cF$ be a family of $N$-tuples $\bal=\a_i,i=1,2,\ldots,N$. Let
\beq{Ztal}{
Z=\inf_{\bal\in\cF}\sum_{i=1}^N\a_iX_i.
}
Let $\s=\sup_{\bal\in\cF}\brac{\sum_{i=1}^N\a_i^2}^{1/2}$ and let $\m$ be a median of $Z$. Then, for all $u>0$,
\beq{Tal1}{
\Pr(|Z-\m|\geq u)\leq 4\exp\set{-\frac{u^2}{4\s^2}}.
}
We need the following claim: let $T_\th$ denote the minimum cost spanning tree with costs $C_\th$. 
\begin{claim}\label{largest}
Let $0\leq \th\leq K/n$. Then $\Pr(\max_{e\in T_\th}W_{e,\th}\geq n^{-9/10})=o(n^{-200})$.
\end{claim}
{\em Proof of claim}
Let $L=n^{-9/10}$. Just consider running Kruskal's greedy algorithm. The probability that the graph is not connected by the time we reach edges of cost $W_{\th,e}\leq L/2$ is certainly $o(n^{-200})$. ($U_e\leq p=L/2-\th$ implies that $W_{\th,e}\leq L/2$ and $\Pr(G_{n,p}$ is not connected) is at most $\sum_{k=1}^{n/2}\binom{n}{k}k^{k-2}p^{k-1}(1-p)^{k(n-k)}=\sum_{k=1}^{n/2}O(n^{-1}(npe^{-np})^k)=o(n^{-200})$.)\\
{\em End of proof of claim}

Let $X_e=n^{9/10}\min\set{C_{\th}(e),n^{-9/10}}$ for $e\in E(K_n)$. The claim implies that if we compute the minimum spanning tree with costs $X_e$, then with probability $1-o(n^{-200})$ we get the same spanning tree as if we used costs $C_\th(e)$, but with cost inflated by a factor $n^{9/10}$. To apply \eqref{Tal1} we let $N=\binom{n}{2}$ and  we let $\bal$ be the 0/1 incidence vectors of the spanning trees of $K_n$. Thus in this case, $Z$ of \eqref{Ztal} is the cost of the minimum spanning tree using costs $X_e$. Now $\s\leq n^{1/2}$ and $u=n^{3/5}$ so that from \eqref{Tal1} we have
\[
\Pr(|Z-\m|\geq n^{3/5})\leq 4\exp\set{-\frac{n^{1/5}}{4}}=o(n^{-200}).
\] 
It follows that 
\beq{eqconc1}{
\Pr(|L_n(\th)-\wh\m|\geq n^{-3/10})=o(n^{-200})
}
 where $\wh\m=\m n^{-9/10}$. It follows immediately that 
\beq{eqconc3}{
\E L_n(\th)\geq \wh\m-2n^{-3/10}.
}
 On the other hand
\beq{eqconc2}{
\E L_n(\th)\leq \wh\m+n^{-3/10}+n^2\Pr(L_n(\th)\geq  \wh\m+n^{-3/10})\leq \wh\m+2n^{-3/10}.
}
So, from \eqref{eqconc1}, \eqref{eqconc3}, \eqref{eqconc2} we get $|L_n(\th)-E L_n(\th)|\leq 4n^{-3/10}$ with probability $1-o(n^{-100})$.
\end{proof}

The fixed-\(\theta\) estimate above implies a uniform estimate on \([0,K/n]\). Indeed, \(L_n(\theta)\) is \((n-1)\)-Lipschitz in \(\theta\), and \(-\theta\alpha n\) is \(\alpha n\)-Lipschitz. Hence \(\phi(\theta)\) is at most \(2n\)-Lipschitz. Take a grid of mesh \(n^{-4}\) in \([0,K/n]\), containing \(O(n^3)\) points. The fixed-\(\theta\) concentration bound, union-bounded over the grid, gives concentration at every grid point with probability \(1-o(n^{-90})\). Lipschitz continuity then gives $\sup_{0\leq \theta\leq K/n}|\phi(\theta)-\E\phi(\theta)|=o(1)$ w.h.p. Therefore \(\max_\theta\phi(\theta)=\max_\theta\E\phi(\theta)+o(1)\) w.h.p. Since the limiting value in Theorem \ref{th1} is bounded away from zero, this absolute $o(1)$ concentration implies $Z_\a\sim \E Z_\a$ w.h.p.

\section{Proof of Theorem \ref{th2}}
The proof here is written for the undirected case $K_n$, but the same argument will work for the complete digraph $\vec K_n$ if we use outgoing edges.
\subsection{A lower bound}
We prove that for every fixed $\eta<\lambda(\a)$,
\[
  \Pr\!\left(Z_\a\le \eta\frac{\log n}{n}\right)\longrightarrow 0.
\]
Let
\[
  t_n:=\eta\frac{\log n}{n},
\]
and let $N_n(\eta)$ be the number of paths from $1$ to $2$ with total weight at most $t_n$ and at most $L$ red edges. A fixed $k$-edge path has weight $\G_k$ where
\[
  \Pr\bigl(\Gamma_k\le t_n\bigr) =\int_0^{t_n}\frac{x^{k-1}e^{-x}}{(k-1)!}\,\dd x  \le \frac{t_n^k}{k!},
\]
where the inequality follows from $e^{-x} \le 1$. 
There are at most $ n^{k-1}$ paths of length $k$ from 1 to 2.  For such a path, the red-edge count has distribution $Bin(k,1/2)$, independently of the weights.  Therefore
\begin{align}
  \E N_n(\eta)  &\le  \sum_{k=1}^{n-1} n^{k-1}\frac{t_n^k}{k!} \Pr\bigl(Bin(k,1/2)\le L\bigr) \nn \\
  &=  \frac1n\sum_{k=0}^{\infty}\frac{(\eta\log n)^k}{k!}\sum_{\ell=0}^L\binom{k}{\ell}2^{-k}\nn\\
&= \frac1n\sum_{\ell=0}^L\frac{1}{\ell!}\bfrac{\eta\log n}{2}^{\ell}\sum_{k=\ell}^\infty\frac{(\eta\log n)^{k-\ell}}{2^{k-\ell}(k-\ell)!}\nn\\
&=\frac1{n^{1-\eta/2}}\sum_{\ell=0}^L\frac{1}{\ell!}\bfrac{\eta\log n}{2}^{\ell}.\label{final}
\end{align}
If $\eta\le 2\a<1$ then \eqref{final} gives
\[
  \E N_n(\eta)\le n^{\eta-1}\to0.
\]
It remains to consider $\eta\geq2\a$. Now the sum in \eqref{final} is dominated by its last term, so
\[
\frac1{n^{1-\eta/2}}\sum_{\ell=0}^L\frac{1}{\ell!}\bfrac{\eta\log n}{2}^{\ell}=O\brac{\frac1{n^{1-\eta/2}}\cdot\frac{1}{L!}\bfrac{\eta\log n}{2}^{L}}=O(n^{-1+\eta/2+\a(1+\log\eta-\log 2\a)})=o(1),
\]
since $\eta<\l(\a)$.
\subsection{An upper bound}
The idea here is to grow shortest path trees from 1 and 2, up to certain depth. Taking account of the coloring, we argue that each tree contains a large number of vertices whose path from the root has small cost and uses fewer than $L/2$ red edges. Then it is simple to argue that w.h.p. we can join two such vertices to make a short path from 1 to 2. 

Consider the analysis of Janson \cite{J2} in the uncolored case. Recall Dijkstra's algorithm DA.
After $k$ iterations there is a rooted {\em shortest path} tree $T_k$ with root $s=1$ such that if
$v$ is a vertex of $T_k$ then the tree path from $s$ to $v$ is a shortest path. Let $d(v)$ be its length.
For $x\notin T_k$ let $d(x)$ be the minimum length of a path $P$ that goes from $s$ to $v$ to $x$ where $v\in T_k$ and the sub-path of $P$ that goes to $v$ is the tree path from $s$ to $v$. If $d(y)=\min\set{d(x):x\notin T_k}$ then $d(y)$ is the length of a shortest path from $s$ to $y$ and $y$ can be added to the tree.

Suppose that vertices are added to the tree in the order $v_1=s,v_2,\ldots,v_n$ and that $Y_j=dist(v_1,v_j)$ for $j=1,2,\ldots,n$. It follows from the memoryless property of the exponential distribution\footnote{$\Pr(Z\geq A+t\mid Z\geq A)=\Pr(Z\geq t)$ for an exponential random variable $Z$}
\[
Y_{k+1} = \min_{\substack{i=1,2,\ldots,k \\{v\neq v_1,\ldots,v_k}}}[Y_i+ X_{v_i,v}]=Y_k + E_k
\]
where $E_k$, being the minimum of $k(n-k)$ independent exponentials, i.e. exponential with mean $\frac{1}{k(n-k)}$ and independent of $Y_k$. This is because $X_{v_i,v_j}$ is distributed as an independent exponential $Z$ conditioned on $Z\geq Y_k-Y_i$. Hence
\beq{ELn}{
\E Y_n = \sum_{k=1}^{n-1}\frac{1}{k(n-k)} = \frac{1}{n}\sum_{k=1}^{n-1}\left(\frac{1}{k}+ \frac{1}{n-k}\right) = \frac{2}{n}\sum_{k=1}^{n-1}\frac{1}{k} = \frac{2\log n}{n}+O(n^{-1}).
}
Now consider what changes when we color the edges. We apply the Dijkstra's algorithm (DA) ignoring the colors and produce a sequence of trees $T_1,T_2,\ldots,T_n$ where $T_k$ has exactly $k$ vertices. A vertex $v$ of $T_n$  is {\em valid} if the path from 1 to $v$ contains at most $L$ red edges after the edge colors are revealed.

Suppose that vertices are added to the tree in the order $v_1=s,v_2,\ldots,v_n$. Then if $i\leq k<j$, the edge cost $X_{v_i,v_j}$ is distributed as an independent exponential $Z_{i,j}$ conditioned on $Z_{i,j}\geq d(v_k)-d(v_i)$.  It follows from the memoryless property of the exponential distribution that $Z_{i,j}=d(v_k)-d(v_i)+E_{i,j}$ where the $E_{i,j}$ are independent exponential mean 1. Then,
\[
 \min_{1\leq i\leq k<j}(d(v_i)+Z_{i,j})=d_k+\min_{1\leq i\leq k<j}E_{i,j}.
\]
It follows that  when we add a new vertex it is connected to a random vertex in the existing tree.
\subsection{A colored Yule process}
We model the growth of the trees $T_1,T_2,\ldots,T_n$ as a {\em Yule process}. Consider a continuous-time branching process started from one root at time $0$.  Each particle lives forever and gives birth at rate 1. At each birth, the child is joined to its parent by an independently colored edge, red with probability $1/2$ and blue with probability $1/2$. In algorithm DA, the particles are the vertices, the root is 1 and when a child is born, it is born to a random parent. Denote the Yule tree constructed by time $t$ as $\cY(t)$.

In both processes when a vertex is added to a tree, it is connected to a uniform random neighbor. (Such trees are called {\em Random recursive trees} and have been well studied.) We are only going to use valid paths of $\cY(t)$. If $v$ is the $k$th vertex added to $\cY(t)$ and the path $P_k$ to $v$ is valid then the expected corresponding Dijkstra cost of $P_k$ is 
\beq{k(n-k)}{
\frac{1}{n}\sum_{i=1}^k\left(\frac{1}{i}+ \frac{1}{n-i}\right) \sim \frac{\log k}{n}+O\bfrac{1}{n},
}
if $k=o(n)$.

Meanwhile the $k$th birth of $\cY$ occurs at expected time $\sim\sum_{i=1}^k\frac{1}{i}\sim \log k$. This is because the time to the next birth when there are $i$ vertices in $\cY$ is the minimum of $i$ independent exponentials with mean $1$. Meanwhile, having explored $k$ vertices, the next increment in cost has distribution $EXP(k(n-k))$, so for $k = o(n)$ the scaled process $nd(v)$ has increments close to those of a rate-$k$ Yule process. The parent of the new vertex is uniform among the $k$ existing vertices, and the edge color is independent of the costs and of the past. Thus the rooted colored tree structure agrees with a colored random recursive tree up to $o(n)$ explored vertices, while distances are Yule times divided by $n$, up to a $1 + o(1)$ factor.

We will prove the following lemma in Section \ref{main1} below;
\begin{lemma}\label{mainshort}
Fix $a,\eta\in(0,1)$, $\gamma\in(0,a/2)$ and $s<h(a,\gamma)$ where
\[
h(a,\gamma)=\frac{a}{2}+\gamma+\gamma\log\!\left(\frac{a}{2\gamma}\right).
\]
In the complete graph $K_{n-o(n)}$ with independent $EXP(1)$  edge costs, run algorithm DA until $N=n^{1-\eta}$ vertices have been discovered. Then w.h.p. there are at least $N^s$ vertices $v\in T_N$ such that the path in $T_N$ from 1 to $v$ has total cost at most $\frac{\left(a+o(1)\right)\log N}{n}$ and uses at most $\gamma\log N$ red edges.
\end{lemma}
Assume the truth of Lemma \ref{mainshort} with $\eta>0$ is sufficiently small.  We choose $a=\l/2+\e,\g=\a/2$ where $\e>0$ is sufficiently small. We argue below that for some $c_1>0$, we have
\beq{hag}{
h(a,\g)\geq \frac12+c_1\e+O(\e^2).
}
Given \eqref{hag}, choose $s,\eta$ so that $1/2 < s(1-2\eta) <s< h(a, \g)$ and fix a set $W$ of size $n^{3/2-s}$. Apply Lemma \ref{mainshort} and run DA on $K_n\sm W$ from vertex 1 until we have explored $N=n^{1-\eta}$ vertices and found $N^s=n^{(1-\eta)s}$ valid vertices $V_1$ as promised by Lemma \ref{mainshort}. Now do the same from vertex 2, running DA as if on $K_{n-N}\sm W$, avoiding the $N=o(n)$ vertices examined for vertex 1 until we obtain $N^{s}$ valid vertices $V_2$. Note that the costs of the edges between $V_1,V_2$ and $W$ are unconditioned.  If there is a path $P=(w_1\in V_1,w\in W,w_2\in V_2)$ using blue edges of cost at most $1/n$ then concatenating the path from 1 to $w_1$ in $V_1$ and the path $P$ and the path from $w_2$ to 2 in $V_2$ gives us (after possibly removing cycles) a path of cost at most $(\l+2\e+o(1))\log n/n$ containing at most $\a\log n$ red edges. Now the probability that there is no such path $P$ can be bounded
\[
(1-(1-(1-\tfrac12(1-e^{-1/n}))^{n^{(1-\eta)s}})^2)^{n^{3/2-s}}=e^{-\Omega(n^{(1-2\eta)s-1/2})}=o(1).
\]
Since $\e$ is arbitrary, Theorem \ref{th2} follows.
\paragraph{Verification of \eqref{hag}}
We first observe that from the definition of $\a$, we have
\[
h\brac{\frac\l2,\frac\a2}=\frac12\brac{\frac{\l}2+\a+\a\log\bfrac{\l}{2\a}}=\frac12.
\]
Then we observe that
\beq{h+}{
\nabla h\brac{\frac\l2,\frac\a2}=\brac{\frac12+\frac\a\l,\log\bfrac\l{2\a}}>0.
}
since $\l>2\a$. (If $\l\leq 2\a$ then its definition implies that $\l/2\geq 1-\a>\a$, contradiction.)
\subsection{Proof of Lemma \ref{mainshort}: Structure of $\cY(t)$}\label{main1}

 We analyze $\cY(t)$, proving the analog of Lemma \ref{mainshort} for the Yule process, which we now state. 
\begin{lemma}\label{mainshort2}
Fix $a\in(0,1)$, $\gamma\in(0,a/2)$ and $s<h(a,\gamma)$ where
\[
h(a,\gamma)=\frac{a}{2}+\gamma+\gamma\log\!\left(\frac{a}{2\gamma}\right).
\]
In the Yule process, the following holds w.h.p. There are at least $N^s$ particles $v$ such that the path  from the root to $v$ appears by time $\left(a+o(1)\right)\log N$ and uses at most $\gamma\log N$ red edges.
\end{lemma}
Lemma \ref{mainshort2} implies  Lemma \ref{mainshort}. Indeed, we can couple the DA exploration process up to step $N$ with the Yule process, where to calculate the DA costs from the Yule costs we divide by $n'=n-o(n)$. In the DA process, at step $k$ the next vertex arrival is distributed as $EXP(k(n'-k))$ and for $k \le N = o(n)$ we have $k(n'-k) \sim kn$. 

For $\theta>0$ define
\[
  S_\theta(t):=\sum_{u\in\mathcal Y(t)} e^{-\theta R_u},  \qquad  \psi(\theta):=\frac{1+e^{-\theta}}{2},
  \qquad  M_\theta(t):=e^{-\psi(\theta) t}S_\theta(t),
\]
where $R_u$ denotes the number of red edges in the path from the root to $u$ in $\cY(t)$.
\begin{lemma}\label{lem:yule_mart}
For every fixed $\theta>0$, $(M_\theta(t))_{t\ge0}$ is a nonnegative martingale, bounded in $L^2$, and it converges almost surely and in $L^2$ to a finite limit $M_\theta(\infty)$ satisfying
\[
  \Pr(M_\theta(\infty)>0)=1.
\]
Consequently, for every fixed $a>0$,
\[
  \frac{\log M_\theta(a\log N)}{\log N}\xrightarrow{\Pr}0.
\]
\end{lemma}
\begin{proof}
The martingale property follows directly from the generator $\cal L$ of the Yule branching process. $\cal L$ acts on a function of the current tree by computing its instantaneous expected rate of change. If a particle \(u\) has red-depth \(R_u=r\), then a blue child is born at rate \(\rho=1/2\) and contributes \(e^{-\theta r}\) to \(S_\theta\), while a red child is born at rate \(\rho\) and contributes \(e^{-\theta(r+1)}\).  Hence
\[
\mathcal L S_\theta=\rho\sum_{u\in\mathcal Y(t)}\left(e^{-\theta R_u}+e^{-\theta(R_u+1)}\right)=\rho(1+e^{-\theta})S_\theta=\psi(\theta)S_\theta.
\]
Therefore, the instantaneous expected rate of change for \(F(t,\cY(t))=M_\th(t)\) is
\[
(\partial_t+\mathcal L)F(t,\cY(t))=-\psi(\theta)e^{-\psi(\theta)t}S_\theta+e^{-\psi(\theta)t}\mathcal L S_\theta=0.
\]
(Here $\partial_t$ is the deterministic contribution due to the factor $e^{-\psi(\theta)t}$ and $\cal L$ is due to the stochastic drift.) 

We use Dynkin's formula \cite{Dyn}  in its time-dependent form.  If \(X_t\) is a continuous-time Markov process with generator \(\mathcal L\) and $F$ is a suitable function, then
\[
\E\left(F(t,X_t)\mid\mathcal F_s\right)=F(s,X_s)+\E\left(\int_s^t(\partial_r+\mathcal L)F(r,X_r)\,dr\mid\mathcal F_s\right).
\]
In particular, if \((\partial_t+\mathcal L)F(t,X_t)=0\), then \(F(t,X_t)=M_\theta(t)\) is a martingale.  

For a second moment calculation, observe that a birth increases the quantity $S_\theta(t)^2$ by $2S_\theta(t)Y+Y^2$.  We have
\[
 \E(Y^2)=\r\sum_{u\in \cY(t)}(e^{-2\th R_u}+e^{-2\th(R_u+1)})= \r S_{2\theta}(t)+\r e^{-2\theta}S_{2\theta}(t) = \psi(2 \th) S_{2 \th}(t).
\]
Thus
\[
  \frac{\dd}{\dd t}\E(S_\theta(t)^2) =2\psi(\theta)\E(S_\theta(t)^2)+\psi(2\th)\E S_{2\th}(t)=2\psi(\theta)\E(S_\theta(t)^2) +\psi(2\theta)e^{\psi(2\theta)t}.
\]
Because $\psi(2\theta)<2\psi(\theta)$ for $\theta>0$, integrating gives us
\[
 \E(S_\theta(t)^2) =e^{2\psi(\th)t}\brac{1+\psi(2\th)\int_0^te^{(\psi(2\th)-2\psi(\th))s}\dd s}\leq e^{2\psi(\th)t}\brac{1+\frac{\psi(2\theta)}{2\psi(\theta)-\psi(2\theta)}} .
\]
So,
\[
  \E(M_\theta(t)^2)  =e^{-2\psi(\theta)t}\E(S_\theta(t)^2)  \le  1+\frac{\psi(2\theta)}{2\psi(\theta)-\psi(2\theta)}<\infty
\]
uniformly in $t$.  Thus, by the Martingale Convergence Theorem, $M_\theta(t)$ converges almost surely and in $L^2$ to a limit $M_\theta(\infty)$. Note that $\E M_\theta(\infty)=\E M_\th(0)=1$.

It remains only to check that this limit is almost surely positive.  Fix $T>0$.  By the branching property,
\[
  M_\theta(\infty)
  =\sum_{u\in\mathcal Y(T)} e^{-\theta R_u-\psi(\theta)T}M^{(u)}_\theta(\infty),
\]
where, conditional on the process up to time $T$, the random variables $M^{(u)}_\theta(\infty)$ are independent copies of $M_\theta(\infty)$.  Let $q:=\Pr(M_\theta(\infty)=0)$.  Since all coefficients in the last sum are positive,
\[
  q=\E(q^{|\mathcal Y(T)|}).
\]
The Yule population satisfies $|\mathcal Y(T)|\ge1$ almost surely and $\Pr(|\mathcal Y(T)|>1)>0$.  Hence, for $0<q<1$, one has $\E( q^{|\mathcal Y(T)|})<q$.  Also $q=1$ is impossible because $\E M_\theta(\infty)=1$.  Therefore $q=0$.

The logarithmic statement follows immediately from almost sure convergence to a finite positive limit.
\end{proof}
Given the above lemma, we can give the proof of Lemma \ref{mainshort2}.

Since $\partial h(a,\gamma)/\partial\gamma=\log(a/2\gamma)>0$, we may choose $\bar\gamma<\gamma$ such that $s<h(a,\bar\gamma)$. Choose $\delta>0$ such that $0<\bar\gamma-\delta<\bar\gamma+\delta<\gamma$. Then with $\d$ sufficiently small and $\theta=\log\left(a/2\bar\gamma\right)>0$, we have $a\psi(\theta)+\theta(\bar\gamma-\delta)>s$. This is because 
\[
  a\psi(\theta)+\theta\bar\gamma  =\frac a2+\bar\gamma+\bar\gamma\log\!\left(\frac{a}{2\bar\gamma}\right)
  =h(a,\bar\gamma).
\]
Let $t=a\log N$, then for any $u>0$,
\[
  \sum_{R_v<(\bar\gamma-\delta)\log N} e^{-\theta R_v}\le  e^{u(\bar\gamma-\delta)\log N}\sum_{v\in\cY(t)} e^{-(\theta+u)R_v} =
  \exp\set{(u(\bar\gamma-\delta)+a\psi(\theta+u))\log N}  M_{\theta+u}(t).
\]
The derivative with respect to $u$ at $u=0$ of $u(\bar\gamma-\delta)+a\psi(\theta+u)-a\psi(\theta)$ is $  \bar\gamma-\delta-ae^{-\theta}/2=\bar\gamma-\delta-\bar\gamma =-\delta$. Thus, for some small fixed $u>0$ and some $\kappa>0$,
\[
  u(\bar\gamma-\delta)+a\psi(\theta+u)
  \le a\psi(\theta)-\kappa.
\]
By Lemma~\ref{lem:yule_mart}, $M_{\theta+u}(t)\le e^{\kappa\log N/2}$ w.h.p.  Therefore w.h.p. 
\[
  \sum_{R_v<(\bar\gamma-\delta)\log N} e^{-\theta R_v}\le  \exp\!\left((a\psi(\theta)-\kappa/2)\log N\right).
\]
Analogously, choose $0<u<\theta$.  Then
\[
  \sum_{R_v>(\bar\gamma+\delta)\log N} e^{-\theta R_v}\le  e^{-u(\bar\gamma+\delta)\log N}\sum_v e^{-(\theta-u)R_v} = \exp\set{(-u(\bar\gamma+\delta)+a\psi(\theta-u))\log N}  M_{\theta-u}(t).
\]
The derivative with respect to $u$ at $u=0$ of $  -u(\bar\gamma+  \delta)+a\psi(\theta-u)-a\psi(\theta)$ is $-\bar\gamma-\delta+ae^{-\theta}/2=-\delta$. So, after possibly decreasing $u$, w.h.p.,
\[
 \sum_{R_v>(\bar\gamma+\delta)\log N} e^{-\theta R_v}\leq  \exp\set{(a\psi(\theta)-\kappa/2)\log N}.
\]
On the other hand, Lemma~\ref{lem:yule_mart} gives that w.h.p.,
\[
  \sum_{v\in\cY(t)} e^{-\theta R_v}=e^{\psi(\theta)t}M_\theta(t)  \ge \exp\set{(a\psi(\theta)-o(1))\log N}.
\]
Therefore w.h.p., 
\beq{window}{
\sum_{R_v\in[\bar\g\pm\d]\log N}e^{-\theta R_v}\geq  \exp\set{(a\psi(\theta)-o(1))\log N}.
}
Consequently, the number of particles in the sum on the LHS of \eqref{window} must be at least
\[
  \exp\set{(a\psi(\theta)+\theta(\bar\gamma-\delta)-o(1))\log N}\geq N^s.
\]
 Lemma \ref{mainshort2} now follows since $\bar\g+\d\leq\g$.
\section{Proof of Theorem \ref{th3}}
In this section, we proceed as for spanning trees and consider a dual problem. Thus let $X$ denote the set of perfect matchings of $K_{n,n}$. Our optimization problem is to compute 
\begin{align}
C^*=\text{Minimum:}&\hspace{0.25in} C(M)\label{matchings}\\
\text{Subject to:}&\hspace{0.25in} M\in X\nn\\
&\hspace{0.25in}|E_R\cap M|\leq \a n.\label{M1}
\end{align}
We analyse this through Lagrangian relaxation. Define the dual function $\f(\th)$ for $\th\geq 0$ by
\beq{eqn:phidef3}
{
\f(\th)=-\th\a n+\min_{M\in X}\sum_{e\in M}(X_e+\th1_{e \in E_R}).
}
Let $Z_e, e \in E(K_{n, n})$ be i.i.d. exponential mean 1, and let $\wh Z_e$ be i.i.d. Bernoulli $1/2$. For each $e\in E(K_{n,n})$ set 
\beq{eqn:Wdef}
{
W_{\th,e} = Z_e + \theta \wh Z_e\text{ and }L_n(\th) = \min_{M \in X} W_\th(M)\text{ where }W_\th(M)=\sum_{e \in M} W_{\th,e}.
}
Then
\[
\E\f(\th)=-\th\a n+\E L_n(\th).
\]
Consider the complete bipartite graph $K_{n,n}$, with bipartition $(A,B)$, where $A=\{a_1, a_2,\ldots, a_n\}$ and $B=\{b_1, b_2, \ldots, b_n\}$, and with edge weights which are independent copies of $W_{\th,e}$. We assume that $a_1,a_2,\ldots,a_n$ is a random permutation of $A$. So $A_r=\set{a_1,a_2,\ldots,a_r}$ is a random $r$-subset of $A$. 

Let $EXP(\l)$ denote an exponential random variable of rate $\lambda$ i.e. $\Pr(EXP(\lambda)\geq x)=e^{-\lambda x}$. We add a special vertex $b^*$ to $B$, with edges to all $n$ vertices of $A$. Each edge adjacent to $b^*$ is assigned an $EXP(\lambda)$ weight independently, $\lambda>0$. 

For $r\geq 1$ we let $M_r$ be the minimum weight matching of $A_r$ into $B$ and $M_r^*$  be the minimum weight matching of $A_r$ into $B^*=B\cup \set{b^*}$. (As $\lambda\to 0$ it becomes increasingly unlikely that any of the extra edges are actually used in the minimum weight matching.) We denote this matching by $M_r^*$ and we let $B_r^*$ denote the corresponding set of vertices of $B^*$ that are covered by $M_r^*$. We let $C(n,r)$ denote the weight of $M_r$.

Define $P(n, r)$ as the normalized probability that $b^*$ participates in $M_r^*$, i.e.
\begin{equation}
 P(n, r) = \lim_{\lambda \rightarrow 0^+} \frac{\Pr(b^*\in B_r^*)}{\lambda}.
\end{equation}
Its importance lies in the following lemma:
\begin{lemma}\label{lem10a}
\begin{equation}
\E(C(n, r) - C(n, r-1)) = \frac{P(n, r)}{r}.
\end{equation}
\end{lemma}
\begin{proof}
Choose $i$ randomly from $[r]$ and let $\widehat{B}_i\subseteq B_r$ be the $B$-vertices in the minimum weight matching of $(A_r\setminus\set{a_i})$ into $B^*$. Let $X=C(n,r)$ and let $Y=C(n,r-1)$. Let $w_i$ be the weight of the edge $(a_i, b^*)$, and let $I_i$ denote the indicator variable for the event that the minimum weight of an  $A_{r}$ matching that contains this edge is smaller than the minimum weight of an $A_{r}$ matching that does not use $b^*$. We claim that $\Pr(I_i)= \Pr(Y_i + w_i < X) + O(\l^2)$, where $Y_i$ is the minimum weight of a matching from $A_r\setminus \set{a_i}$ to $B$. 
Indeed, if $(a_i, b^*) \in M_r^*$ then $w_i < X - Y_i$. Conversely, if $w_i < X - Y_i$ and no other edge from $b^*$ has weight smaller than $X - Y_i$, then $(a_i, b^*)\in M_r^*$, and when $\lambda \to 0$, the probability that there are two distinct edges from $b^*$ of weight smaller than $X - Y_i$ is of order $O(\lambda^2)$. Indeed, let ${\mathcal F}$ denote the existence of two distinct edges from $b^*$ of weight smaller than $X$ and let ${\mathcal F}_{i,j}$ denote the event that $(a_i,b^*)$ and $(a_j,b^*)$ both have weight smaller than $X$. 

Then, 
\beq{no2}{
\Pr({\mathcal F})\leq n^2\E_X(\max_{i,j}\Pr({\mathcal F}_{i,j}\mid X))=n^2\E((1-e^{-\lambda X})^2)\leq n^2\lambda^2\E(X^2),
}
and since $\E(X^2)$ is finite and independent of $\lambda$, this is $O(\lambda^2)$.

Note that $Y$ and $Y_i$ have the same distribution. They are both equal to the minimum weight of a matching of a random $(r-1)$-set of $A$ into $B$. As a consequence, $\E(Y)=\E(Y_i)=\frac{1}{r}\sum_{j\in A_r}\E(Y_j)$. Since $w_i$ is $EXP(\lambda)$ distributed, as $\lambda\to 0$ we have from \eqref{no2} that 
\mults{
P(n,r)= \lim_{\lambda \rightarrow 0} \brac{\frac{1}{\lambda}\sum_{j\in A_r}\Pr(w_j<X-Y_j)+O(\lambda)}=\\
\lim_{\lambda \rightarrow 0}\E\brac{ \frac{1}{\lambda}\sum_{j\in A_r} \brac{1- e^{-\lambda(X-Y_j)}}}
=\sum_{j\in A_r}\E(X - Y_i) = r\E(X - Y).
}
\end{proof}
We now proceed to estimate $P(n,r)$. Fix $r$ and assume that $b^*\notin B_{r-1}^*$. Note that 
\begin{equation}\label{eqn:augpath}
    \text{$M_r^*$ is obtained from $M_{r-1}^*$ via a single augmenting path with $a_r$ as an endpoint. In particular, $B_{r-1}^* \subseteq B_r^*.$}
\end{equation}
Indeed, clearly the symmetric difference $M_{r-1}^* \Delta M_r^*$ contains some $M_{r-1}^*$-augmenting path with $a_r$ as an endpoint, since $M_r^*$ covers $a_r$ and $M_{r-1}^*$ does not. Since both matchings cover $a_1, \ldots, a_{r-1}$, any other component of $M_{r-1}^* \Delta M_r^*$ must be either a path or cycle with an even number of edges. But both matchings are supposed to minimize cost, so they could not disagree on such a component. Thus there are no such components and \eqref{eqn:augpath} is proved.

Suppose that $M_{r}^*$ is obtained from $M_{r-1}^*$ by finding an augmenting path $P=(a_{r},\ldots,a_{\s},b_\t)$ from $a_{r}$ to $B\setminus B_{r-1}$ of minimum additional weight. We condition on (i) $\s$, (ii) the costs of all edges other than $(a_\s,b_j),b_j\in B\setminus B_{r-1}$ and \\
(iii) $\min\set{w(a_\s,b_j):b_j\in B\setminus B_{r-1}}$. With this conditioning $M_{r-1}=M_{r-1}^*$ will be fixed and so will $P'=(a_{r},\ldots,a_\s)$. Let $m=n-r+2$ and $X_1,X_2,\ldots,X_{m-1}$ be independent copies of $W_{\th,e}$  and let $X_m$ be exponential with rate $\l$. We compute the probability that $X_m$ is the smallest of them if we condition on the value of $\min\set{X_1,X_2,\ldots,X_m}$. Indeed,

\begin{claim}
 For any $\g>0$, we have
 \beq{eq:conditional}
 {
 \Pr(X_m=\min\set{X_1,\ldots,X_m}\mid \min\set{X_1,\ldots,X_m}=\g)= \begin{cases}
\frac{\lambda}{\lambda + m-1} & \g \ge \th\\
    \frac{\lambda}{\lambda + \frac{m-1}{1+e^\g}}& \g < \th
\end{cases}
 }
\end{claim}
\begin{proof}
Let $g_i(x)$ be the probability density function for $X_i$, and let $G_i(x)$ be its cumulative density function. We have

\begin{align*}
g_1(x) = \cdots = g_{m-1}(x)&=\begin{cases}\frac12e^{-x}+\frac12e^{-(x-\th)}&x\geq \th.\\ \frac12e^{-x}&x<\th.\end{cases}\\
G_1(x) = \cdots = G_{m-1}(x)&=\begin{cases}1-\frac12e^{-x}-\frac12e^{-(x-\th)}&x\geq \th.\\ \frac12 - \frac12e^{-x}&x<\th\end{cases}\\
g_m(x) &= \lambda e^{-\lambda x}\\
G_m(x) &= 1- e^{-\lambda x}
\end{align*}
 Thus 
 \begin{align*}
&\Pr(X_m=\min\set{X_1,\ldots,X_m}\mid \min\set{X_1,\ldots,X_m}=\g)\\
&=\frac{g_m(\g)(1-G_1(\g))^{m-1}}{g_m(\g)(1-G_1(\g))^{m-1}+(m-1)g_1(\g) (1-G_m(\g))(1-G_1(\g))^{m-2}}\\
&=\frac{1}{1+(m-1)\frac{g_1(\g)(1-G_m(\g))}{g_m(\g)(1-G_1(\g))}}\\
& =  \begin{cases}
\frac{\lambda}{\lambda + m-1} & \g \ge \th\\
    \frac{\lambda}{\lambda + \frac{m-1}{1+e^\g}}& \g < \th
\end{cases}
\end{align*}

\end{proof}

Define
\[
\F_r:=\Pr(X_m=\min\set{X_1,\ldots,X_m}) = \Pr(b^*\in B_r^*\mid b^*\notin B_{r-1}^*).
\]
If $f, F$ denote the pdf and cdf of $\min\{X_1, \ldots, X_m\}$, then
\[
\F_r=\int_{\g=0}^\th\frac{\l f(\g)}{\l+\frac{m-1}{1+e^\g}}\dd\g +\frac{\l(1-F(\th))}{\l+m-1}.
\]
Now assuming that $\g \le \th$, we have
\begin{align*}
1-F(\g) = \Pr(\min\{X_1, \ldots, X_m\}\geq \g)&=\bfrac{1+e^{-\g}}{2}^{m-1}e^{-\lambda \g} \\
&=\bfrac{1+e^{-\g}}{2}^{m-1}+O(\l).
\end{align*}
Note that above and for the rest of this section, asymptotics are as $\l \rightarrow 0$ and all other variables are held constant. Taking derivatives, for  $\g\le \th$ we have 
\begin{align*}
  f(\g)&=\frac12(m-1)e^{-\g}\bfrac{1+e^{-\g}}{2}^{m-2}e^{-\lambda \g}+\bfrac{1+e^{-\g}}{2}^{m-1}\lambda e^{-\lambda \g}\\
  & = \frac12(m-1)e^{-\g}\bfrac{1+e^{-\g}}{2}^{m-2}+O(\lambda )
\end{align*}

So,
\begin{align*}
  \F_r&=\int_{\g=0}^\th\frac{\l\cdot \frac12(m-1)e^{-\g}\bfrac{1+e^{-\g}}{2}^{m-2}}{\l+\frac{m-1}{1+e^\g}}\dd\g +\frac{\l\bfrac{1+e^{-\th}}{2}^{m-1}}{\l+m-1}+O(\l^2)\\
  &=\frac{\l}{2^{m-1}}\int_{\g=0}^\th(1+e^{-\g})^{m-1}\dd\g+\frac{\l(1+e^{-\th})^{m-1}}{2^{m-1}(m-1)}+O(\l^2)\\
  &=\frac{\l}{2^{m-1}}\int_{x=e^{-\th}}^1\frac{(1+x)^{m-1}}{x}\dd x+\frac{\l(1+e^{-\th})^{m-1}}{2^{m-1}(m-1)}+O(\l^2)\\
&=\l\Psi(r,\th)+O(\l^2),
\end{align*}
where we have defined (recalling $m=m(r) = n-r +2$)
\[
\Psi(r,\th):= \frac{1}{2^{m-1}}\brac{\int_{x=e^{-\th}}^1\frac{(1+x)^{m-1}}{x}\dd x+\frac{(1+e^{-\th})^{m-1}}{m-1}}.
\]
  It follows from \eqref{eqn:augpath} that
\begin{align}
 P(n, r) =  \lim_{\l\to0^+}\l^{-1}\Pr(b^*\in B_r^*)& = \lim_{\l\to0^+}\l^{-1}\brac{1-\prod_{i=1}^r\Pr(b^* \notin B_r^* \mid b^* \notin B_{r-1}^*)}\nn\\
 &=\lim_{\l\to0^+}\l^{-1}\brac{1-\prod_{i=1}^r(1-\F_i)}\nn\\
 &= \lim_{\l\to0^+}\l^{-1}\brac{1-\prod_{i=1}^r\brac{1-\l\Psi(i,\th)+O(\l^2)}}\nn\\
  &=\sum_{i=1}^r\Psi(i,\th)\nn\\
  &=\sum_{i=1}^r\frac{1}{2^{n-i+1}}\brac{\int_{x=e^{-\th}}^1\frac{(1+x)^{n-i+1}}{x}\dd x+\frac{(1+e^{-\th})^{n-i+1}}{n-i+1}}\label{Ups}\\
&=:\Upsilon(r,\th).\nn
\end{align}
Now using  $L_n(\th) = C(n, n)$, Lemma \ref{lem10a}, and $P(n, r) = \Upsilon(r, \th)$ we get
\beq{Efth}{
\E\f(\th) = -\th \a n + \E C(n, n) =-\th \a n + \sum_{r=1}^n  \E(C(n, r) - C(n, r-1)) =-\th\a n+\sum_{r=1}^n\frac{\Upsilon(r,\th)}{r}.
}
So the relaxation of \eqref{matchings} gives an expected lower bound of $\f(\th^*)$ where $\th^*$ is the solution to
\begin{align*}
  \a n&= \sum_{r=1}^n\frac{1}{r}\frac{\partial \Upsilon(r,\th)}{\partial \th}\\
      &=\sum_{r=1}^n\frac{1}{r}\sum_{i=1}^r\frac{(1+e^{-\th})^{n-i}}{2^{n-i+1}}\\
      &=\frac{(1+e^{-\th})^n}{2^{n}}\sum_{r=1}^n\frac{2^{r}-(1+e^{-\th})^{r}}{r(1+e^{-\th})^r(1-e^{-\th})}\\
      &\sim ne^{-\psi/2}\sum_{r=1}^n\frac{e^{\psi r/2n}-1}{r\psi}\\
  &\sim n\psi^{-1} e^{-\psi/2}\int_{x=0}^{\psi/2}\frac{e^{x}-1}{x}\dd x
\end{align*}
assuming that $\th=\psi/n,\psi=O(1)$ (this assumption is justified in Subsection \ref{sec:boundingtheta}). Note that while $\psi > 0$ for any $\a < 1/2$, the above expression approaches $n/2$ when $\psi \searrow 0$. This verifies our estimate for $\E Z_\a$.

\subsection{Concentration of $\f$}
Concentration of $\f(\th)$ around its expectation can be justified as in Section \ref{concsec}. All we need is verification that w.h.p.  the optimal matchings do not contain edges of weight more than $n^{-9/10}$. This is a fairly standard calculation that depends on the expansion of $K_{n,n,p}$. We will use the following lemma:
\begin{lemma}\label{Mn-1}
In $K_{n,n,p}$ with $p=n^{-10/11}$, with probability $1-o(n^{-200})$, for every matching $M$ of size $n-1$, there is an augmenting path of length at most 53.
\end{lemma}
\begin{proof}
Indeed suppose that the bipartition of vertices in $K_{n,n}$ is $A,B$ and $N_p$ denotes neighborhood in $K_{n,n,p}$. Then, 
\begin{align*}
&\Pr(\exists S\subseteq A:|S|\leq n^{10/11} \log n,|N_p(S)|\le np|S|/2\log n)\\
&\leq \sum_{s=1}^{n^{10/11}\log n}\binom{n}{s}\binom{n}{n-nps/2\log n}(1-p)^{(n-nps/2\log n)s}\\
& \leq \sum_{s=1}^{n^{10/11}\log n}\brac{\frac{ne}{s}}^s \brac{\frac{ne}{nps/2\log n}}^{nps/2\log n} e^{-(n-nps/2\log n)ps}
 \leq \sum_{s=1}^{n^{10/11}\log n}  e^{-nps/3}=o(n^{-200}).
\end{align*}
Suppose that $u\in A,v\in B$ are the two vertices not covered by $M$. The set $A_1 \subseteq A$ of vertices reachable from $u$ by an alternating path of length at most 26 has $|A_1| \ge n/3$, with the required probability. Indeed, starting at $u$ we can find the tree containing all possible such alternating paths by iteratively finding the neighborhood of the current bottom level of the tree, then taking the matching edges covering those vertices, and repeating. Every time we add two levels to the tree, the number of vertices at the bottom level expands by a factor of at least $np/3 \log n = n^{1/11}/3 \log n$ (where we replaced a ``2'' with a ``3'' to account for any vertices that were already in the tree). This lasts until the bottom layer has $n^{10/11} \log n$ vertices, which happens by the time we reach level 26. Two more levels give us another expansion by factor $n^{1/11}/3 \log n$, bringing us to $n/3$ vertices in the bottom level. Similarly, there is a set of at least $n/3$ vertices $B_1$ in $B$ that are reachable from $v$ by an alternating path of length at most 26. Then w.h.p., whatever the sets $A_1,B_1$,  there is a non-matching edge  joining $A_1$ and $B_1$, giving an alternating path $P$ from $u$ to $v$ of length at most 53. Indeed, the probability that there exist $A_1,B_1$ not joined by at least $n+1$ edges is at most $2^{2n}e^{-\Omega(n^2p)}=o(n^{-200})$.
\end{proof}
Suppose the minimum cost perfect matching $M^*$ contains an edge $e=(u,v)$, $u \in A, v \in B$,  of weight at least $n^{-9/10}$.  Remove $e$ to obtain near perfect matching $M$. Note that $n^{-10/11}$ is the probability that an edge is blue and has cost at most $(2+o(1))n^{-10/11}$. Using the path $P$ promised by Lemma \ref{Mn-1}, we can find a matching with a cost $n^{-9/10}-O(n^{-10/11})$ less than $M^*$, a contradiction. 
\subsection{Bounding $\th^*$}\label{sec:boundingtheta}
We need to justify the assumption that $\th^*=O(1/n)$. The dual function $\f(\th)=L_n(\th)-\th\a n$ and w.h.p. $L_n(\th)=O(1)$, independent of $\th$, given that we can just use blue edges for a perfect matching. We have to maximise this over $\th\geq 0$.
This rules out $\th\geq K/n$ for some constant $K>0$, as we know that $\max\f(\th)\geq \z(2)$.

We next estimate $\f(\th^*)$ assuming $\th^*=\psi/n$. For this, we go back to \eqref{Ups} and \eqref{Efth}. First of all
\begin{align}
\sum_{r=1}^n\sum_{i=1}^r\frac1r \cdot \frac{(1+e^{-\psi/n})^{n-i+1}}{2^{n-i+1}(n-i+1)}&=\sum_{r=1}^n\sum_{i=1}^r\frac1r\cdot \frac{(1-\frac{\psi}{2n}+O(n^{-2}))^{n-i+1}}{2^{n-i+1}(n-i+1)}\nn\\
&=\sum_{r=1}^n\sum_{i=1}^r\frac{e^{-(n-i+1)\psi/2n}+O(1/n)}{r(n-i+1)}\label{intbit}\\
&=\sum_{i=1}^n\frac{e^{-i\psi/2n}}{i}\sum_{r=n-i+1}^n\frac{1}{r}+o(1)\nn\\
&=-\int_{x=0}^1\frac{e^{-\psi x/2}\log(1-x)}{x}\dd x+o(1).\label{afterbit}
\end{align}
Then we have
\begin{align*}
\int_{x=e^{-\psi/n}}^1\frac{(1+x)^{n-i+1}}{x}\dd x&=\int_{y=0}^{1-e^{-\psi/n}}\frac{\brac{1-\frac{y}{2}}^{n-i+1}}{1-y}\dd y\\
&=\int_{y=0}^{1-e^{-\psi/n}}\frac{e^{-(n-i+1)y/2}}{1-y}\dd y+O(1/n)\\
&=\frac{2}{n-i+1}(1-e^{-(n-i+1)\psi/2n})+O(1/n).
\end{align*}
Then, using \eqref{intbit}, we obtain
\[
\sum_{r=1}^n\sum_{i=1}^r\frac1r \cdot \frac{2}{n-i+1}(1-e^{-(n-i+1)\psi/2n})\sim \frac{\pi^2}{3}+2\int_{x=0}^1\frac{e^{-\psi x/2}\log(1-x)}{x}\dd x+o(1).
\]
Combined with \eqref{afterbit}, we get the expression in \eqref{PMcost}.

Now $\f(\th^*)$ only gives a lower bound on the minimum cost of a perfect matching. We discuss a matching upper bound.  Let $\cM^*(\th)$ denote the set of matchings that minimise $W_{\th}$. According to \cite{Ge}, $\f(\th^*)$ is the objective value of the  solution to an LP relaxation of \eqref{matchings} and this solution is a convex combination of matchings in $\cM^*(\th)$.  Thus, if $\cM^*(\th^*)$ contains a unique matching $M^*$ then $C(M^*)=\f(\th^*)$ and we are done.   
\begin{lemma}\label{cl1}
With probability one, $|\cM^*(\th)|\leq 2$ for all $\th > 0$.
\end{lemma}
\begin{proof}
Suppose that there are three distinct members of $\cM^*(\th)$ that minimise  $W_{\th}$. This implies that there are three distinct perfect matchings $M_i,i=1,2,3$ with equal values for  $C(M_i)+\th|M_i\cap E_R|$. Now if $|M_1\cap E_R|=|M_2\cap E_R|$ then $C(M_1)=C(M_2)$, an event of probability zero. We can therefore assume that all the $|M_i\cap E_R|$ are distinct. But this implies, that 
\beq{eqx}{
\th=\frac{C(M_1)-C(M_2)}{|M_2\cap E_R|-|M_1\cap E_R|}=\frac{C(M_1)-C(M_3)}{|M_3\cap E_R|-|M_1\cap E_R|}
}
or
\[
(C(M_1)-C(M_2))(|M_3\cap E_R|-|M_1\cap E_R|)=(C(M_1)-C(M_3))(|M_2\cap E_R|-|M_1\cap E_R|),
\]
 an event of probability zero.
\end{proof}
So assume there are $M_1,M_2\in \cM^*(\th^*)$  that satisfy $C(M_1)+\th^*|M_1\cap E_R|=C(M_2)+\th^* |M_2\cap E_R|=L_n(\th^*)$. Following \cite{BBGS}, we prove
\begin{lemma}\label{lemm}
Given $M_1,M_2$ we can construct a matching $M$ with $n-1$ edges for which $C(M)\leq \f(\th^*)+O(1/n)$ and $|M\cap E_R|\leq \a n$. 
\end{lemma}
\begin{proof} 
Now we cannot have $\d_1= \min\set{|M_1\cap E_R|,|M_2\cap E_R|}-\a n>0$ else $\f(\th^*+\d_2)>\f(\th^*)$ for some small $\d_2>0$. This follows from Lemma \ref{cl1}. Let $\xi=\min\set{W_{\th^*}(M)-W_{\th^*}(M_1):M\neq M_1,M_2}>0$ and let $\d_2=\xi/2n$. Then we have $C_{\th^*+\d_2}(M)\geq C_{\th^*+\d_2}(M_1)+\xi-n\d_2>C_{\th^*+\d_2}(M_1)$ for $M\neq M_1,M_2$ so that $M_1,M_2$ minimise $C_{\th^*+\d_2}$. And then $\f(\th^*+\d_2)-\f(\th^*)=(|M_1\cap E_R|-\a n)\d_2>0$. 

Now $\th^*>0$ w.h.p. This is because $\f(0)$ is the minimum cost of a matching ignoring the color constraint. Then w.h.p. this matching $M$, which uniquely minimises cost, satisfies $|M\cap E_R|\sim n/2$ and then $\f(\d_3)=C(M)+\d_3(1/2-o(1)-\a)n>C(M)=\f(0)$ for $\d_3$ sufficiently small, contradiction. 

Given this, we can rule out $\d_4= \max\set{|M_1\cap E_R|,|M_2\cap E_R|}-\a n<0$. For then we could argue as for $\d_1>0$ that we can reduce $\th^*$ by a small positive amount, less than $\th^*$, and increase $\f$.  

If $|M_i \cap E_R | = \a n$ for some $i$, then $M_i$ itself is feasible and satisfies $C(M_i) = L_n(\th^*)-\th^*|M_i\cap E_R|=L_n(\th^*)-\th^*\a n=\f(\th^*)$, so there is nothing to prove.

Assume then that $|M_1\cap E_R|<\a n<|M_2\cap E_R|$. Let $K=\set{e_1,e_2,\ldots,e_k}=M_1\oplus M_2$, where $M_1$ is the set of odd subscript edges.  Assume that each cycle of $M_1 \oplus M_2$ appears as some contiguous subset $\{e_j, e_{j+1}, \ldots, e_{j'}\}$ of $K$. Let $a_i=\d(e_i)W_{\th^*, e_i}$ where $\d(e_i)=1$ for $e_i\in M_1$ and -1 otherwise. Then 
\[
W_{\th^*}(M_1)-W_{\th^*}(M_2)= \sum_{i=1}^ka_i=0
\]
 and so there exists $\ell$ such that (below, subscripts are taken modulo $k$)
\beq{1}{
\sum_{j=1}^ta_{\ell+j}\geq 0\text{ for }t=0,1,\ldots,k-1.
}
Indeed, it suffices to take $\ell = \text{argmin}\left\{\sum_{j=1}^\ell a_j: 1 \le \ell \le k \right\}$. This is the content of the gasoline lemma of Lov\'asz \cite{L}, Problem 3.21. 

For $t\geq 0$ let $X_t=M_1\cup \set{e_{\ell+j}:j=1, \ldots, t,\,\ell+j\text{ is even}}\setminus \set{e_{\ell+j}:j=1, \ldots, t,\,\ell+j\text{ is odd}}$. Let $\t=\max\set{t:|X_t\cap E_R|\leq \a n}$. Then we must have $\ell+\t$ odd, since $X_{t+1} \subseteq X_{t}$ when $\ell+t$ is even. Let $M= X_\t \setminus e_\ell$ if $\ell$ is odd, and $M= X_\t \setminus e_{\ell+1}$ if $\ell$ is even. Note that $M$ is a matching, being obtained from $M_1$ via an alternating path. Also, $|M\cap E_R|\leq |X_\t\cap E_R|\leq {\a n}$ and $|M|=n-1$. Note that \eqref{1} implies that 
\[
W_{\th^*}(M)\leq W_{\th^*}(X_\t)\leq W_{\th^*}(M_1).
\]
So,
\begin{align*}
C(X_\t)&=W_{\th^*}(X_\t)-\th^* |X_\t\cap E_R|\\
&=W_{\th^*}(X_\t)-\th^* {\a n}+\th^*({\a n}- |X_\t\cap E_R|) \\ 
&\leq W_{\th^*}(M_1)-\th^* {\a n}+\th^*\\
&=\f(\th^*)+\th^*.
\end{align*}
The maximality of $\t$ implies that $e_{\ell + \t + 1}\in E_R$. So,
\beq{up1}{
C(M)\leq C(X_\t)\leq \f(\th^*)+\th^*=\f(\th^*)+O(1/n).
}
Furthermore, by construction,
\beq{upC1}{
|M\cap E_R|\leq {\a n}.
}
\end{proof}
At this point, we need to deal with the fact that $|M|=n-1$. 
\begin{lemma}\label{Mn-1x}
W.h.p. for every matching $M$ of size $n-1$ there is an augmenting path that creates a matching $M'$ with (i) $|M'|=|M|+1$, (ii) $C(M')\leq C(M)+O(n^{-10/11})$ and (iii) $|M'\cap E_R|\leq|M\cap E_R|$.
\end{lemma}
\begin{proof}
Apply Lemma \ref{Mn-1} where $K_{n,n,p}$ is defined by the blue edges of cost at most $2n^{-10/11}$. The promised augmenting path will produce a perfect matching. The increase in cost will be $O(n^{-10/11})$ and no red edges will be added.
\end{proof}
Combining Lemmas~\ref{lemm} and~\ref{Mn-1x}, w.h.p. there is a perfect matching $M'$ satisfying $|M'\cap E_R|\le \alpha n$ and $C^*\leq C(M')\le \phi(\theta^*)+O(n^{-10/11})+O(1/n)=\phi(\theta^*)+o(1)$. Weak duality gives $C^*\ge \phi(\theta^*)$, and hence $C^*=\phi(\theta^*)+o(1)$ w.h.p.  Thus we have that w.h.p. 
\[
\f(\th^*)\leq Z_\a\leq \f(\th^*)+o(1).
\]
So, $Z_\a\sim \f(\th^*)$ w.h.p. and then the concentration of $\phi$ around its mean shows that $Z_\a$ is also concentrated around the mean defined in the statement of Theorem~\ref{th3}.

\section{Proof of Theorem \ref{th4}}
Karp's patching algorithm first solves the problem of finding a minimum weight perfect matching $M^*$ (Assignment Problem) in $K_{n,n}$. We associate this solution to a cycle cover, i.e. a set of vertex disjoint directed cycles in the complete digraph $\vec K_n$. (Assume that the vertices of $K_{n,n}$ are $A=\set{a_1,a_2,\ldots,a_n}$ and $B=\set{b_1,b_2,\ldots,b_n}$. Edge $\set{a_i,b_j}$ of $M^*$ is replaced by the directed edge $(i,j)$ of $\vec K_n$. This produces a cycle cover. Note that $\set{a_i,b_i}$ gives rise to a loop at vertex $i$. Note also that this shows that $Z_\a(PM)\leq Z_\a(ATSP)$.)

We instead have to solve the problem of finding a minimum weight perfect matching under edge color constraints.  We can solve the associated Linear Programming relaxation and then applying Lemma \ref{lemm} find an asymptotically optimal matching $M^*$ that satisfies the color constraints. It follows from the proof of Lemma \ref{Mn-1} that no edge of the matching $M^*$ has $X_e\geq \eta_0=C_0n^{-1/3}$ for some $C_0>0$. So, w.h.p. we can construct $M^*$ by ignoring all edges $E_{Large}$ of cost greater than $\eta_0$. 

One can argue by symmetry as was done by Karp \cite{Karp} that $M^*$ has a uniform distribution and so its associated cycle cover has $m=O(\log n)$ cycles. (Choose the constrained optimum matching using a label-equivariant tie-breaking rule, or choose uniformly among all constrained optima. Since the joint law of colors and costs is invariant under permutations of the two vertex classes, the resulting perfect matching is uniformly distributed over all perfect matchings.) Let these cycles be $K_1,K_2,\ldots,K_m$ where $|K_i|\geq |K_{i+1}|$. Let $E_1,E_2,\ldots,E_{m-1}$ be a random partition of $E_{Large}$ into $m-1$ almost equal size subsets. We proceed by patching the cycles together. Given cycles $K,K'$ we patch them together by finding edges $e=(a,b)\in K,f=(c,d)\in K'$ and then replacing $e,f$ by $e'=(a,d),f'=(c,b)$. This gives us a new cycle with vertex set $V(K)\cup V(K')$. We call $e,f$ a patching pair if $(a,d),(c,b)$ are blue and both edges have cost in $[\eta_0,2\eta_0]$.  (If one cycle is a loop on vertex $a$ then we modify the construction appropriately, replacing an edge $(b,c)$ in the larger cycle by the path $b,a,c$.) Different pairs of arcs in $K_i,K_j$ give rise to disjoint patching pairs with independent edge costs. We ignore the saving associated with deleting $e,f$ and only look at the extra cost $C(e')+C(f')$ incurred. 

We use $E_1$ to find a cycle to patch into $K_1$, creating a cycle $\wh K$. In general, for $i\geq 2$, we use $E_i$ to find a cycle to patch into $\wh K$, increasing its size. Assuming that we search over all possible patching pairs involving $\wh K$ and unabsorbed cycles, the number of possible patching pairs $|\wh K|(n-|\wh K|)\geq n-1$ holds for every $i$. So the probability we fail to find the $m-1$ needed patches is at most $m(1-\Omega(\eta_0^2/m^2))^{n-1}=o(1)$. The additional cost associated with this construction is $O(mn^{-1/3})=o(1)$ and only blue edges are added and so the red count does not increase. Thus the resulting Hamilton cycle is feasible and has cost $Z_\alpha(PM)+o(1)$. Since $Z_\alpha(PM)\le Z_\alpha(ATSP)$, this proves Theorem~\ref{th4}.
\section{General distributions}\label{distributions}
We will now consider the situation where the costs/weights $X_e$ are independent copies of a continuous random variable $C$. We assume that $C$ has a continuous density $f$ and satisfies
\begin{enumerate}[(i)]
\item $f(x)=a+bx+O(x^2)$ for $0\leq x\leq L$, where $a,b$ are constants and $aL\geq 1$. 
\item $\Pr(C\geq x)\leq c_1e^{-c_2x}$ for constants $c_1,c_2>0$.
\item To avoid some pathologies, we will also assume that there is a constant $M$ such that $f(x)\leq aM$ for $x\geq 0$.
\end{enumerate}
The prime examples are the uniform $[0,1]$ distribution $U[0,1]$ ($a=1,b=0,M=1$) and the exponential mean 1 distribution $EXP(1)$ ($a=1,b=-1,c_1=c_2=M=1$). 

After replacing $C$ by $aC$, the new density at zero is 1. Thus it suffices to prove the
result for $a = 1$; the original costs are obtained by dividing the limiting constants by $a$.

First consider Theorem \ref{th1} where we assumed uniform costs. Let $F(x)=\Pr(C\leq x)=x+\frac b2x^2+O(x^3)$ and let $\wh C$ be uniformly distributed. Then we can take $C=F^{-1}(\wh C)$ and $C$ will have the correct distribution. If $\wh C=y$ then $C=x=y+O(y^2)$. Now Lemma \ref{largest} implies that w.h.p. the maximum edge cost in our solution is $O(n^{-9/10})$, we see that the optimal costs using $C,\wh C$ differ by $O(n\times n^{-9/5})=o(1)$. Thus Theorem \ref{th1} holds with more generality.

Now consider Theorems \ref{th2}, \ref{th3} where we use $EXP(1)$ for costs/weights. This time we take $\wh C=-\log(1-F(C))$ so that $\Pr(\wh C\leq x)=\Pr(-\log(1-F(C))\leq x)=\Pr(F(C)\leq1 -e^{-x})=1-e^{-x}$. Notice that $-\log(1-F(x))=x+O(x^2)$. In Theorem \ref{th2} we only use edges of cost $O(\log n/n)$ and in Theorem \ref{th3} we only use edges of weight $O(n^{-9/10})$ and so we can proceed as for Theorem \ref{th1}.

In the case of Theorem \ref{th4} we use the same $\wh C$ as for Theorems \ref{th2} and \ref{th3}. In which case the possible increases in error are $O(n\cdot n^{-9/5})$ for the Assignment Problem and $O(\log n\cdot n^{-2/3})$ for the patching of the cycles.

\end{document}